\documentclass[10pt,conference]{IEEEtran}

\usepackage{mathrsfs}
\usepackage{graphics} 
\usepackage{amsmath} 
\usepackage{amssymb}  
\usepackage{amsthm}
\usepackage{cite}
\usepackage{bm}
\usepackage{acronym}
\usepackage{paralist}
\usepackage{float}
\usepackage{color}
\usepackage{epstopdf}
\usepackage{multicol}
\usepackage{tikz}
\usepackage{hyperref}
\usepackage{mathtools}
\usepackage{subcaption}
\usepackage{tabularx,ragged2e}

\usepackage{graphicx}
\usepackage{diagbox}

\usepackage{soul}
\usepackage[dvipsnames]{xcolor}

\makeatletter
\hypersetup{colorlinks=true}
\AtBeginDocument{\@ifpackageloaded{hyperref}
  {\def\@linkcolor{blue}
  \def\@anchorcolor{red}
  \def\@citecolor{red}
  \def\@filecolor{red}
  \def\@urlcolor{red}
  \def\@menucolor{red}
  \def\@pagecolor{red}
\begingroup
  \@makeother\`%
  \@makeother\=%
  \edef\x{%
    \edef\noexpand\x{%
      \endgroup
      \noexpand\toks@{%
        \catcode 96=\noexpand\the\catcode`\noexpand\`\relax
        \catcode 61=\noexpand\the\catcode`\noexpand\=\relax
      }%
    }%
    \noexpand\x
  }%
\x
\@makeother\`
\@makeother\=
}{}}
\makeatother

\usepackage[linesnumbered,ruled,vlined]{algorithm2e}

\SetKwComment{Comment}{$\triangleright$\ }{}
\SetKw{KwAnd}{and}
\SetKw{KwOr}{or}
\SetKw{KwNot}{not}

\newtheorem{Theorem}{Theorem}

\newtheorem{Problem}{Problem}

\newtheorem{Remark}{Remark}

\newtheorem{Assumption}{Assumption}

\DeclareMathOperator*{\argmin}{\textrm{argmin}}

\newcommand{\bequ}{\begin{eqnarray}}
\newcommand{\eequ}{\end{eqnarray}}

\IEEEoverridecommandlockouts

\def\BibTeX{{\rm B\kern-.05em{\sc i\kern-.025em b}\kern-.08em
    T\kern-.1667em\lower.7ex\hbox{E}\kern-.125emX}}

\begin{document}

\title{\Large \bf Hybrid topology control: a dynamic leader-based distributed edge-addition and deletion mechanism}

\author{Kunal Garg \and Xi Yu\thanks{KG is with the School of Engineering for Matter, Transport and Energy, XY is with the School of Manufacturing Systems and Networks at the Arizona State University.  Email: \texttt{\{kgarg24, xyu\}@asu.edu}.}}

\maketitle

\begin{abstract}
Coordinated operations of multi-robot systems (MRS) require agents to maintain communication connections to accomplish team objectives. However, maintaining the connections imposes costs in terms of restricted robot mobility, resulting in suboptimal team performance. In this work, we consider a realistic MRS framework in which agents are subject to unknown dynamical disturbances and experience communication delays. Most existing works on connectivity maintenance use consensus-based frameworks for graph reconfiguration, where decision-making time scales with the number of nodes and requires multiple rounds of communication, making them ineffective under communication delays. To address this, we propose a novel leader-based decision-making algorithm that uses a central node for efficient real-time reconfiguration, reducing decision-making time to depend on the graph diameter rather than the number of nodes and requiring only one round of information transfer through the network. We propose a novel method for estimating robot locations within the MRS that actively accounts for unknown disturbances and the communication delays. Using these position estimates, the central node selects a set of edges to delete while allowing the formation of new edges, aiming to keep the diameter of the new graph within a threshold. We provide numerous simulation results to showcase the efficacy of the proposed method. 
\end{abstract}


\section{Introduction}


Complex mission tasks, such as surveillance of deep-sea regions, often involve multi-robot systems (MRS) operating without human intervention or a central coordination system \cite{leeAutonomousApproachObserving2017,loncar_heterogeneous_2019}. As a result, the robots must remain connected to each other so that any task-relevant information can be shared with the MRS team, especially in GPS-denied operations where losing contact with the team may result in loss of the robotic agent. The network topology for communication is required to evolve as the robots make progress towards their objective, without losing structural properties essential for efficient coordination, such as a small graph diameter for quick data transmission \cite{stefanov2011design}. Thus, a key challenge in MRS coordination is controlling how the graph topology evolves, preserving connectivity while maintaining an information flow pattern that remains useful throughout deployment \cite{kantaros2016global,gielis2022critical}.

As an illustrative scenario, consider a team of underwater robots on a deep-sea mission, devoid of human intervention and GPS signals, operating amid unknown water currents. Underwater communication is commonly dominated by acoustic devices with limited bandwidth and range, leading to significant latency and nontrivial packet loss \cite{bresciani_cooperative_2021,ma2025adaptive}. Moreover, underwater communication is often broadcast and strongly affected by multipath \cite{tan2011survey, fang2019precision}, making it difficult to infer the true origin of a received message without additional labeling or processing overhead. 
Furthermore, complex mission tasks require these robots to continuously move to achieve the team objective, which may necessitate changing the communication topology with time. 
In particular, some communication edges might need to be removed to enable MRS to make progress towards the team objective, e.g., if an edge causes deadlock for the complete MRS, it must be disconnected so that the MRS can move \cite{garg2024deadlock}.

Existing solutions for connected MRS can be categorized as \textit{centralized} and \textit{distributed} \cite{sabattini2014decentralized,lin2021online,yang2024decentralized}. Centralized methods are communication-heavy and prone to single-point failures. 
On the other hand, decentralized approaches often rely on achieving consensus for decision-making, which requires frequent information exchange and iterative agreement steps that implicitly presume timely, repeated communication \cite{zavlanos2008distributed,yu2021resilient}. In an asynchronous communication setting with delays, deciding which existing edge can be dropped safely without losing graph connectivity may require multiple rounds of back-and-forth exchanges, delaying decisions and potentially leading to failures or increasing operational costs.

The impact of delay compounds in large-scale MRS settings that require multi-hop communication. When one robot’s message must traverse several hops before reaching a distant teammate, the information is already delayed upon arrival. There is work on state estimation over a network in the presence of time delays \cite{guarro2018state}, but such observer-based techniques are limited to linear systems. 
This motivates the design of a novel topology control mechanism that can work with delayed, uncertain, and bandwidth-constrained information exchange. 
A decision on edge deletion requires the costs of edge maintenance to be communicated from all nodes in the graph to the entity making the decision. In a consensus-based framework, such as max-consensus (see \cite{zavlanos2008distributed}), communication delays slow convergence and require multiple rounds of information exchange. To balance robustness and efficiency, in this work, we adopt a novel amalgamation of a centralized and a decentralized framework that leverages the broadcast communication protocol. 

We propose a \textit{hybrid} topology control framework for MRS operating under delayed and uncertain state information, where the decision is made by a central node after collecting information from all the nodes, but the central node is not fixed and is allowed to change as the graph topology evolves; hence, the usage of the term hybrid (see Figure \ref{fig:overview figure}). 
Furthermore, instead of active communication, we use a broadcast framework in which messages are available to all neighbors. This framework aims to maintain guaranteed connectivity while reducing coordination latency and preserving effective information flow as the team moves in communication-constrained environments, with underwater acoustics serving as a representative motivating example. The main contributions of the paper are as follows:
\begin{enumerate}
    \item We propose a hybrid mechanism for delay-aware topology reconfiguration under multi-hop communication with time delays. Compared to consensus-based approaches that require multi-round communication for convergence, the proposed mechanism requires only one round for decision-making and one for transmitting the decisions.
    \item Unlike existing connectivity-preserving approaches that assume perfect state knowledge, we model disturbances in the system dynamics and design a mechanism for \textit{safe} topology decisions under stale, uncertain information. 
    \item Our proposed topology reconfiguration decisions account for the worst-case information propagation delay by bounding the resulting graph's diameter and switching the central node after topology reconfiguration.
\end{enumerate}

\begin{figure}
    \centering
    \includegraphics[width=0.95\linewidth]{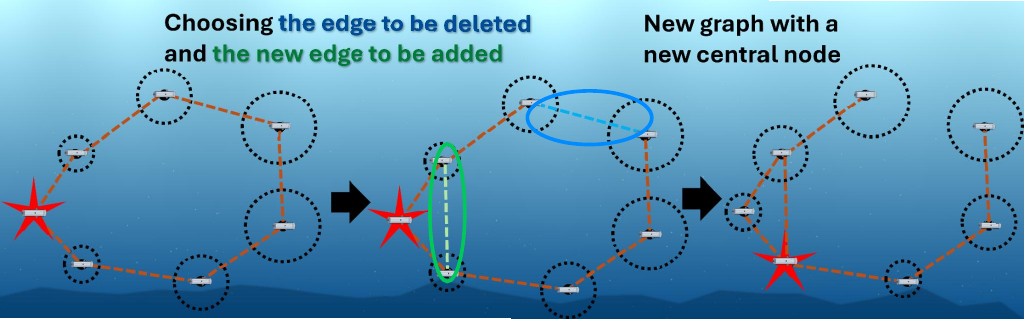}
    \caption{Hybrid topology control framework: the central node decides {\color{blue}\textbf{the edge to be deleted}} and {\color{PineGreen}\textbf{the edge to be added}} to the base graph. Once the new graph is formed, a possibly new central node is chosen. The position uncertainties of each node for the central node are shown through dotted circles.}
    \label{fig:overview figure}
\end{figure}

\section{Preliminaries}\label{sec: Preliminaries}
\textit{Notations}: The set of real numbers is denoted as $\mathbb R$. A circular ball centered at a location $p\in \mathbb R^3$ and of radius $r\geq 0$ is denoted as $\mathbb B_r(p)$. The floor and ceiling functions are denoted as $\lceil \cdot \rceil$ and $\lfloor \cdot\rfloor$, respectively. Given a time-varying connectivity graph $\mathcal G(t) = (\mathcal V, \mathcal E(t))$ on the network among the agents, $\mathcal V = \{1, 2, \cdots, N\}$ denotes the set of nodes and $\mathcal E(t)\subset\mathcal V\times\mathcal V$ denotes the set of edges, where $(i, j)\in\mathcal E(t)$ implies that information can flow from agent $j$ to agent $i$. In this work, we assume broadcast-based communication, which leads to an undirected graph, so that $(i,j)\in \mathcal E(t)\implies (j,i)\in \mathcal E(t)$. Let $|S|$ denote the cardinality of a set $S$ if $S$ is a discrete set and its $n-$ dimensional volume, e.g., length, area, 3D volume, etc., if $S\subset \mathbb R^n$ is a continuous set. 

The graph $\mathcal G$ is said to be connected at time $t$ if there is a path between each pair of agents $(i, j), i, j\in\mathcal V$ at $t$ \cite{mesbahi2010graph}. One method of checking the connectivity of the MRS is through the Laplacian matrix, defined as $L(\mathcal E(t)) \coloneqq \mathcal D(t) - \mathcal A(t)$, where $A(t)$ is the adjacency matrix at time $t$ given as 
\begin{align*}
    \mathcal A_{ij}(t)= \begin{cases}
        1 & (i, j)\in \mathcal E(t), \\
        0 & \text{otherwise},
    \end{cases}
\end{align*}
and $\mathcal D(t)$ is the degree matrix at time $t$ defined as 
\begin{align*}
    \mathcal D_{ij}(t) = \begin{cases}
        \sum_{j}\mathcal A_{ij}(t) & i = j,\\
        0 & \text{otherwise}.
    \end{cases}
\end{align*}
From \cite[Theorem 2.8]{mesbahi2010graph}, the MRS is connected at time $t$ if and only if the second smallest eigenvalue of the Laplacian matrix is positive, i.e., $\lambda_2(L(\mathcal G(t))) > 0$.

Given a connected graph $\mathcal G = (\mathcal V, \mathcal E)$, the eccentricity $e_{\mathcal G}(v)$ of a node $v\in \mathcal V$ is defined as the maximum length of the shortest path to any other node in the graph, i.e.,   $e_{\mathcal G}(v) = \max_{j\in \mathcal V}d_{\mathcal G}(v, j)$, where $d_{\mathcal G}(i, j)$ denote the \textit{graph-distance} between the nodes $i, j\in \mathcal V$, defined as the length of the shortest path between the nodes $i, j$.\footnote{When clear from the context, we omit the subscript $\mathcal G$ from the distance and the eccentricity functions for the sake of brevity.}
The diameter $D(\mathcal G)$ of the graph $\mathcal G$ is defined as the maximum value of the inter-nodes distances, i.e., $D(\mathcal G) = \max_{i, j\in \mathcal V}d(i, j)$, or equivalently, $D(\mathcal G) = \max_{v\in \mathcal V}e(v)$ \cite{chung1989diameters}. The radius of the graph, on the other hand, is defined as the smallest eccentricity of the graph, i.e., $\text{rad}(\mathcal G) = \min_{v\in \mathcal V}e(v)$. As shown in \cite[Ch. 1]{diestel2025graph}, the relationship $\text{rad}(\mathcal G)\leq D(\mathcal G)\leq 2 \text{rad}(\mathcal G)$ holds between the radius and the diameter of a graph.
As a result, we have that $\text{rad}(\mathcal G)\in \left[\lceil\frac{D(\mathcal G)}{2}\rceil, D(\mathcal G)\right]$. A node $v_c$ is called \textit{central} if its largest distance from any other vertex is the smallest, i.e., $v_c = \argmin_{v\in \mathcal V}e(v)$. 

\section{Problem formulation}\label{sec: Problem}
We consider a multi-robot system (MRS) consisting of $N$ robotic agents labeled $\{1, 2, \cdots, N\}$, where the dynamics of the $i$-th robot is given as
\begin{align}\label{eq: sys dyn}
    \dot x_i = f(x_i) + g(x_i)u_i + d(t, x_i),
\end{align}
where $x_i\in \mathbb R^n$ denote the state and $u_i\in \mathcal U_i\subset\mathbb R^m$ the input for the $i-$th robot, $f, g$ the dynamics functions, and $d$ the unknown external disturbance or the unmodeled dynamics. Under a feedback control policy $\pi_i$ for agent $i$, the closed-loop system is given by:
\begin{align}\label{eq: sys closed loop}
    \dot x_i = F(x_i) + d(t, x_i),
\end{align}
where $F(x_i) = f(x_i) + g(x_i)\pi_i$, whose solutions are denoted as $\phi_i(\cdot, x_i(0); \pi_i, d):\mathbb R\to \mathbb R^n$.\footnote{When clear from the context, we suppress the arguments of the solution $\phi$ for the sake of brevity.} Let $p_i\in \mathbb R^3$ denote the position vector of the $i-$th robot and $R>0$ be the communication radius within which it can share or receive information, with $P = [p_1^T, p_2^T, \cdots, p_N^T]^T$, the conglomerate of the positions of the MRS. The agents need to maintain connections with each other to build the team's knowledge, communicate the mission objectives, and coordinate to accomplish the team's objectives. Define a time-varying connectivity graph $\mathcal G(t) = (\mathcal V, \mathcal E(t))$ on the network among the agents. Here, $\mathcal V = \{1, 2, \cdots, N\}$ denotes the set of nodes and $\mathcal E(t)\subset\mathcal V\times\mathcal V$ denotes the set of edges, where $(i, j)\in \mathcal E(t)$ only if $p_{ij}\coloneqq \|p_i(t)-p_j(t)\|\leq R$ at time $t$. 
Next, let $c_{ij}(t)$ be the cost associated with the edge $(i, j)$. This cost can represent the mobility restriction it imposes on the associated nodes to maintain the connection. We can define the distance-based cost for an edge $(i, j)$ as
\begin{align}\label{eq: cost distance}
    c_{ij}(t) = \begin{cases}
        0 & p_{ij}(t) \leq \rho_m R, \\
        c_{max}(p_{ij}(t)-\rho_m R) & \rho_m R\leq p_{ij}(t)\leq R, \\
        \infty, & p_{ij}(t^+)>R, 
    \end{cases}
\end{align}
for some $\rho_m\in (0, 1)$ and $c_{max}>0$, so that the cost increases as the \textit{physical} displacement between the robots grow beyond a threshold $\rho_m R$ and is assigned as $\infty$ if an existing edge breaks at time $t$, denote as $p_{ij}(t^+)>R$. In practice, this threshold can be determined based on the communication protocol so that the communication delay remains below $\delta T$. 

The same cost function can also be used to model the hindrance of maintaining an edge. Given a team objective function taking the value $\Phi(\mathcal G(t), P(t))$\footnote{When clear from the context, we suppress the argument $P(t)$ from the objective function.} for a given graph configuration $\mathcal G(t)$ and MRS positions $P(t)$ at some time $t\geq 0$, the cost of an edge $c_{ij}(t)$ can be described as the magnitude by which the team objective can decrease by deleting an edge, e.g., due to increase in mobility of the MRS and increased capability of the team in making progress. Such an edge $(i, j)\in \mathcal E(t)$ can be labeled as a \textbf{bottleneck} edge for the MRS progress at the time $t$. 
Here, the parameter $\rho_m$ in \eqref{eq: cost distance} can dictate when an edge becomes a bottleneck. Note that the cost function $c_{ij}(t) = 0$ if $(i, j)\notin \mathcal E(t)$, i.e., there is no cost associated with a pair of nodes that do not have a direct edge between them. It is possible that a pair of nodes are within each other's communication region but do not have a communication link. As a result, in that case, the assigned cost for such a pair is still zero. We define the total cost for maintaining the connections in graph $\mathcal G(t)$ at a time $t$ as
    $c_T(\mathcal G(t)) = \sum_{i, j\in \mathcal  V}c_{ij}(t) + \mathbb I[\lambda_2(L(\mathcal E(t))>0]$,
where the indicator function $\mathbb I$ is defined as
    $\mathbb I[a] = 0$ when $a$ is true and $\infty$ otherwise,
so that the total cost of a \textit{disconnected} graph is $+\infty$. Let $\mathcal N_i(t)$ denote the set of neighbors of the node $i$ at time $t\geq 0$ that are within its communication region, i.e., $\mathcal N_i(t) = \{j\in \mathcal V\; |\; p_{ij}(t)\leq R\}$ and $\overline{\mathcal E}(t) = \{(i, j)\; i\in \mathcal V, j\in \mathcal N_i(t)\}$ denote the set of communication edges which need not be the same as the \textit{base} edge set $\mathcal E(t)$. Let $\overline{\mathcal G}(t) = (\mathcal V, \overline{\mathcal E}(t))$ denote resulting the communication graph. 


\begin{Problem}\label{prob: main problem}
    Given a connected graph $\mathcal G(t)$ and a monotonically decreasing total cost budget function $C:\mathbb R\to \mathbb R$, find a set of edge to be deleted  $\mathcal E_d\subset \mathcal E(t)$ and a set of edges to be added $\mathcal E_f\cap \mathcal E(t)=\emptyset$ to the base graph so that for the new graph $\mathcal G'(t) = (\mathcal V, \mathcal E'(t)) =  (\mathcal V, (\mathcal E(t)\setminus \mathcal E_d)\cup\mathcal E_f)$, the following conditions are satisfied for all $t\geq 0$:
    \begin{enumerate}
        \item The total cost of the connectivity maintenance remains within the threshold, $\sum_{(i, j)\in \mathcal E'(t)}c_{ij}(t)\leq C(t)$;
        \item The diameter of the graph remains below a given threshold, i.e., $D(\mathcal G'(t))\leq \tau_D$ for all $t\geq 0$.
    \end{enumerate}
\end{Problem}


\section{Proposed method}\label{sec: proposed method}
\subsection{Communication protocol}
In this work, we assume a \textit{broadcasting}-based communication model, where each robot $i$ broadcasts information at time $t$ that can be received by any neighboring robot $j\in \mathcal N_i(t)$ within a distance $R$ of the broadcasting robot. 
Given a communication medium, such as radio (for ground or aerial applications) or sonar (for underwater applications), the speed $v$ of the medium dictates the time as the total time taken to transfer a message of \textit{temporal length} $\delta T_l$ to a distance $d$ is 
    $\delta T = \frac{d}{v} + \delta T_l$.
Note that the upper-bound on this time delay is given by $\delta T \leq \frac{R}{v} + \delta T_M$, where $\delta T_M>0$ denotes the maximum allowed temporal length of the communication messages. In this work, we assume all transmitted messages to be of the same temporal length $\delta T_M>0$ for the sake of the simplicity of the exposition. 
Each communication packet originating from node $l\in \mathcal V$ contains a message consisting of $M^l = \left\{p_i^l, t_i^l\right\}_{i=1}^N$, where $p_i^l\in \mathbb R^3$ denotes the location of the node $i\in \mathcal V$ at time $t_i^l$, where $t_i^l\leq t$ is the timestamp of the latest information available to node $l$ at the current time $t$ about node $i$. For larger networks, the message can be truncated to the $ k$-nearest neighbors to satisfy bandwidth requirements. Since each node $k\in \mathcal V$ broadcasts such a message $M^k$, the information for each node $i\neq l$ is determined by the latest time of receipt of the information about node $i$ at node $l$:
\begin{align*}
    t_i^l & = \max\{t^k_i, \; k\in \mathcal N_l(t)\}, \\
    p_i^l & = \{p^k_i\; |\; \{p^k_i, t^k_i\}_{k\in \mathcal N_l(t)}, \;  t^k_i = t_i^l\}. 
\end{align*}
In plain words, since each node $l$ receives the location $p^k_i$ of the other nodes $i$ as known to its neighbor nodes $k\in \mathcal N_l(t)$, it keeps the latest available information through $t_i^l$ to maintain the most up-to-date positional information about the MRS. The communication delay leads to a mismatch in the actual location $p_i(t)$ of a robot $i\in \mathcal V$ at a time $t$ and its location $p_i^l$ known to another robot $l\in \mathcal V$ in the MRS. In the next section, we discuss a method of estimating the region of confidence for the estimated locations.   

\subsection{Position estimation}
The nodes represent the dynamical robots whose motion is governed by \eqref{eq: sys dyn}. In a cooperative setting, it is fair to assume that the MRS agents know the base control policy $\pi_i$ for each of the robots $i\in \mathcal V$. Hence, the main source of error in the estimated location $p_i^l$ from the actual position $p_i$ of robot $i\in \mathcal V$ by any other robot $l\in \mathcal V$ is the disturbance $d$ in the dynamics \eqref{eq: sys closed loop}. In order to get a practically useful and less conservative bound on this error, we assume that the control policy $\pi_i$ renders a predefined reference trajectory $x_{i, r}(\cdot)$ exponentially stable for the closed-loop system \eqref{eq: sys closed loop} in the absence of the disturbance $d$ for each robot $i\in \mathcal V$. 
\begin{Assumption}\label{assum: stable policy}
    For each robot $i\in \mathcal V$, given a reference trajectory $x_{i, r}:\mathbb R\to \mathbb R^n$ satisfying $\dot x_{i, r} = F(x_{i,r})$, the control policy $\pi_i$ is such that the closed-loop solution $\phi_i(\cdot) = \phi_i(\cdot,x_i(0); \pi_i, 0) $ with $d \equiv 0$ satisfies:
    \begin{align}
        \big(\phi_i(t)-x_{i, r}(t)\big)^T& \big(F(\phi_i(t))-F(x_{i,r}(t)\big)& \nonumber \\
        \leq & -\lambda \|\phi_i(t)-x_{i, r}(t)\|^2, 
    \end{align}
    for all $t\geq 0$ where $\lambda>\frac{1}{2}$. 
\end{Assumption}

\begin{Remark}
While we acknowledge that the construction of such a reference trajectory or the corresponding policy $\pi_i$ is not straightforward, the focus of this work is not on the low-level control design. We refer the interested readers to \cite{zhang2024gcbf+} for a learning-based approach on distributed control design with state constraints, and to \cite{garg2023learning} for an overview of various methods for control design for MRS.     
\end{Remark}
We can state the following result under mild assumptions on the boundedness of the disturbance $d$ in \eqref{eq: sys closed loop}. For the sake of brevity, let us denote the \textit{nominal} solution, i.e., without disturbance $d$, of \eqref{eq: sys closed loop} $\phi_i(\cdot, x_i(0); \pi_i, 0)$ as $\phi_{i, 0}(\cdot)$ and the \textit{perturbed} solution, i.e., with disturbance $d$, $\phi(\cdot, x_i(0); \pi_i, d)$ as $\phi_{i, d}(\cdot)$. 

\begin{Theorem}\label{thm: error bound}
    Assume that the disturbance $d$ in \eqref{eq: sys closed loop} is uniformly bounded as $\|d(t, x)\|\leq d_M$ for all $t\geq 0$, $x\in \mathbb R^n$, for some $d_M\geq 0$. Let $e_d(t) \coloneq  \|\phi_{i, 0}(t) - \phi_{i, d}(t)\|$. If $\phi_{i, 0}(0) = \phi_{i, d}(0) = x_{i, r}(0)$, then, under Assumption \ref{assum: stable policy}, it holds that
        $e_d(t)\leq \frac{\sqrt{2}}{2\lambda-1}d_M^2(1-e^{-(\lambda-\frac{1}{2})t})$ for all $t\geq 0$.
\end{Theorem}
\begin{proof}
    
    Consider a function $V(t) = \frac{1}{2} \|\phi_{i, 0}(t)-x_{i, r}(t)\|^2+\frac{1}{2} \|\phi_{i, d}(t)-x_{i, r}(t)\|^2$. It is easy to  Its time derivative along the solutions of \eqref{eq: sys closed loop} reads
    \begin{align*}
        \dot V(t) = & \big(\phi_{i, 0}(t)-x_{i, r}(t)\big)^T\big(F(\phi_{i, 0}(t)-F(x_{i, r}(t)\big)\\
        & + \big(\phi_{i, d}(t)-x_{i, r}(t)\big)^T\big(F(\phi_{i, d}(t)-F(x_{i, r}(t)\big)\\
        & + \big(\phi_{i, d}(t)-x_{i, r}(t)\big)^Td(t, \phi_{i,d}). 
    \end{align*}
    Under Assumption \ref{assum: stable policy}, it follows that
    \begin{align*}
        \dot V(t)\leq & -\lambda \|\phi_{i, 0}(t)-x_{i, r}(t)\|^2-\lambda \|\phi_{i, d}(t)-x_{i, r}(t)\|^2\\
        & + \|\phi_{i, d}(t)-x_{i, r}(t)\|\|d(t, \phi_{i,d}\|.
    \end{align*}
    Using Cauchy-Schwarz inequality, we obtain:
    \begin{align*}
      \dot V(t) & \leq -(\lambda-\frac{1}{2})V(t) + \frac{1}{2}\|d(t, \phi_{i, d})\|^2 \\
      & \leq -(\lambda-\frac{1}{2})V(t) + \frac{1}{2}d_M^2,
    \end{align*}
    where the last inequality follows from the uniform bound on the disturbance $d$. Using the fact that $V(0) = 0$ for the given initial conditions, from Comparison Lemma \cite[Lemma 3.4]{khalil2002nonlinear}, it follows that:
    \begin{align*}
        V(t)\leq \frac{1}{2\lambda-1}d_M^2(1-e^{-(\lambda-\frac{1}{2})t}) \quad \forall t\geq 0.
    \end{align*}
    Using the triangle inequality, we have that:
    \begin{align*}
        e_d(t) & = \|\phi_i(t, x_i(0); \pi_i, 0)-\phi_i(t, x_i(0); \pi_i, d)\|\\
        & \leq \|\phi_{i, 0}(t)-x_{i, r}(t)\|+\|\phi_{i, d}(t)-x_{i, r}(t)\|\\
        & \leq \sqrt{2 V(t)} \leq \sqrt{\frac{2(1-e^{-(\lambda-\frac{1}{2})t})}{2\lambda-1}}d_M,
    \end{align*}
    which completes the proof. 
\end{proof}
Under the assumption of a cooperative setting, the \textit{nominal} trajectory of a robot $i\in \mathcal V$ can be computed by any other robot $l\in \mathcal V$ for $t\geq t_i$ based on the information of $x_i(t_i)$. Based on Theorem \ref{thm: error bound}, the error in the position estimate $e_i^l(t) = \|p_i^l-p_i(t)\|$ an be bounded as
\begin{align*}
    e_i^l(t) \leq \sqrt{\frac{2(1-e^{-(\lambda-\frac{1}{2})(t-t_i^l)})}{2\lambda-1}}d_M.
\end{align*}
As a result, the position estimate error is smallest when the robots use the latest available information, corresponding to the maximum timestamp $t_i^l$. Define the bound in the right-hand side of the above inequality for a given time difference $\delta t\geq 0$ as:
\begin{align}
    \epsilon_{\delta t} \coloneqq \sqrt{\frac{2(1-e^{-(\lambda-\frac{1}{2})\delta t})}{2\lambda-1}}d_M
\end{align}
Now, consider a node $i\in \mathcal V$ located at $p_i(t)$ at time $t\geq 0$ with a set of neighbors $\mathcal N_i(t)$. For each $l\in \mathcal N_i(t)$, the node $i$ receives the information $\{p_l, t_l\}_{l\in \mathcal N_i}$ as the position $p_l$ transmitted by the neighbor at time $t_l$. Based on time delay $\delta t_l = t-t_l$, the location $p_l(t)$ of the node $l$ at the time $t$ when the node $l$ receives the message must lie within the ball $\mathbb B_r(p_i(t))$ where $r = v~\delta t_l$. Furthermore, based on the error bound computed above, we know that $p_l(t)\in \mathbb B_{\rho_m}(p_l(t_l))$ where $\rho_m = \epsilon_{\delta t_l}$. Hence, it holds that
\begin{align}\label{eq: pos est 1 hop}
    p_l^i(t)\in P_l^i(t) \coloneqq \mathbb B_{v(\delta t_l)}(p_i(t)) \cap \mathbb B_{\epsilon_{\delta t_l}}(p_l(t_l)).
\end{align}
The above estimate is valid for all the one-hop neighbors $l\in \mathcal N_i(t)$ for any node $i\in \mathcal V$. For multi-hop neighbors with shortest path-length at least 2 along the labeled path $i = 1, 2, \cdots, m = l$, $m\geq 3$, the estimated position of a node $l\in \mathcal V$ by any other node $i\in \mathcal V$ can be computed recursively as:
\begin{align}\label{eq: pos est multi hop}
    p_{m}^{1}(t) \in P_{m}^{1}(t) \coloneqq & \left(\bigcup_{p\in P_{m-1}^1(t^1_{m-1})}\mathbb B_{v \delta t_{m-1}}(p)\right) \nonumber \\
    & \quad \bigcap \mathbb B_{\epsilon_{\delta t_m}}(p_m(t^1_m)),
\end{align}
with $P_{2}^{1}(t_2^1)$ given in \eqref{eq: pos est 1 hop}, where $\delta t_i$ is the communication delay between the nodes $i+1$ and $i$, $i\leq m-1$. 
The main logic of the position estimation in \eqref{eq: pos est multi hop} is that the node $n_{i+1}$ must be within a distance of the node $n_i$ dictated by the communication delay $\delta t_i$. This method helps \textit{shrink} the estimation error from Theorem \ref{thm: error bound}, as illustrated in Figure \ref{fig: pos est error}.

\begin{figure}[t!]
    \centering
    \begin{subfigure}[b]{0.5\linewidth}
        \centering
        \includegraphics[height=1.2in]{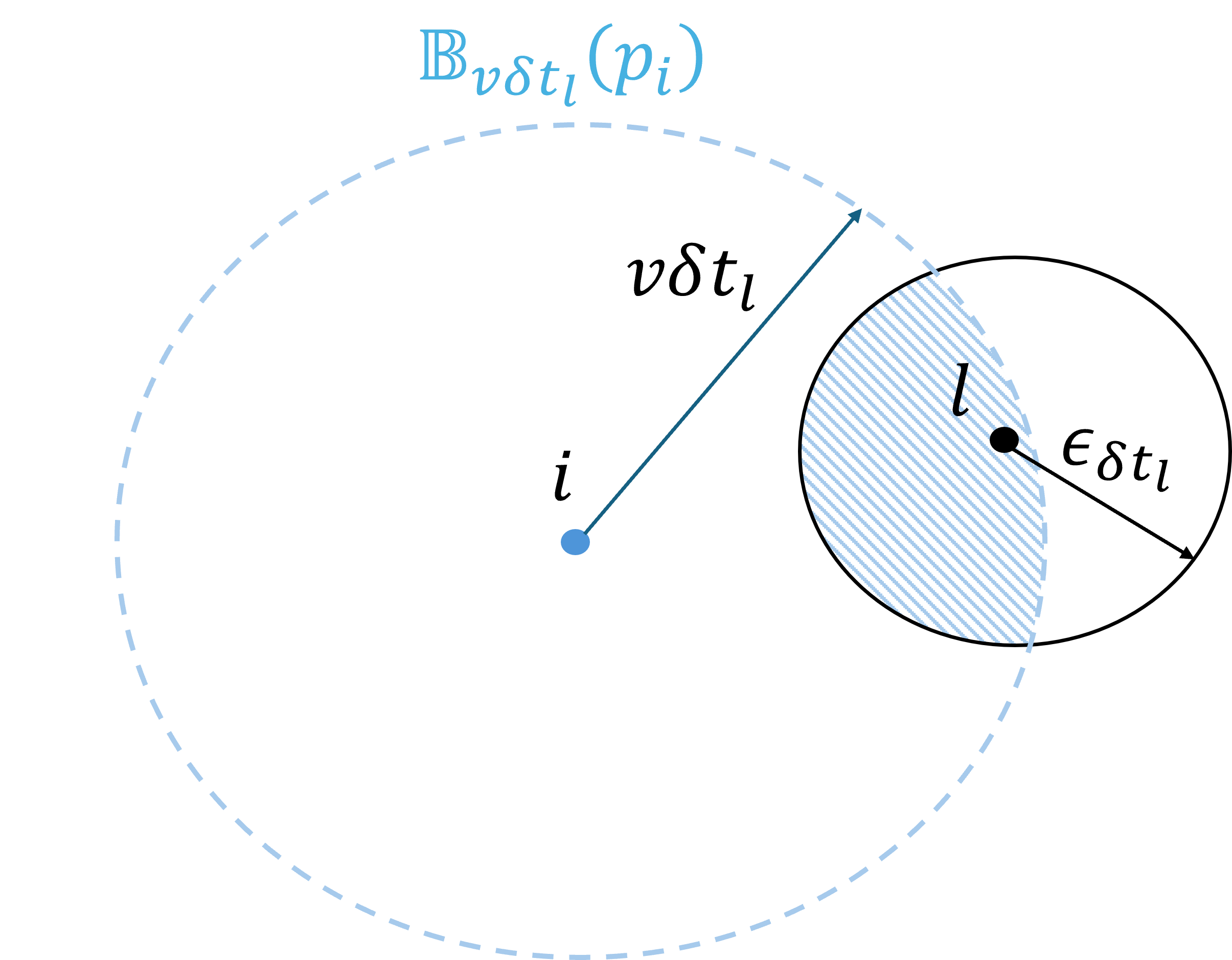}
        \label{fig: pos est error 1}
    \end{subfigure}%
    ~ 
    \begin{subfigure}[b]{0.5\linewidth}
        \centering
        \includegraphics[height=1.2in]{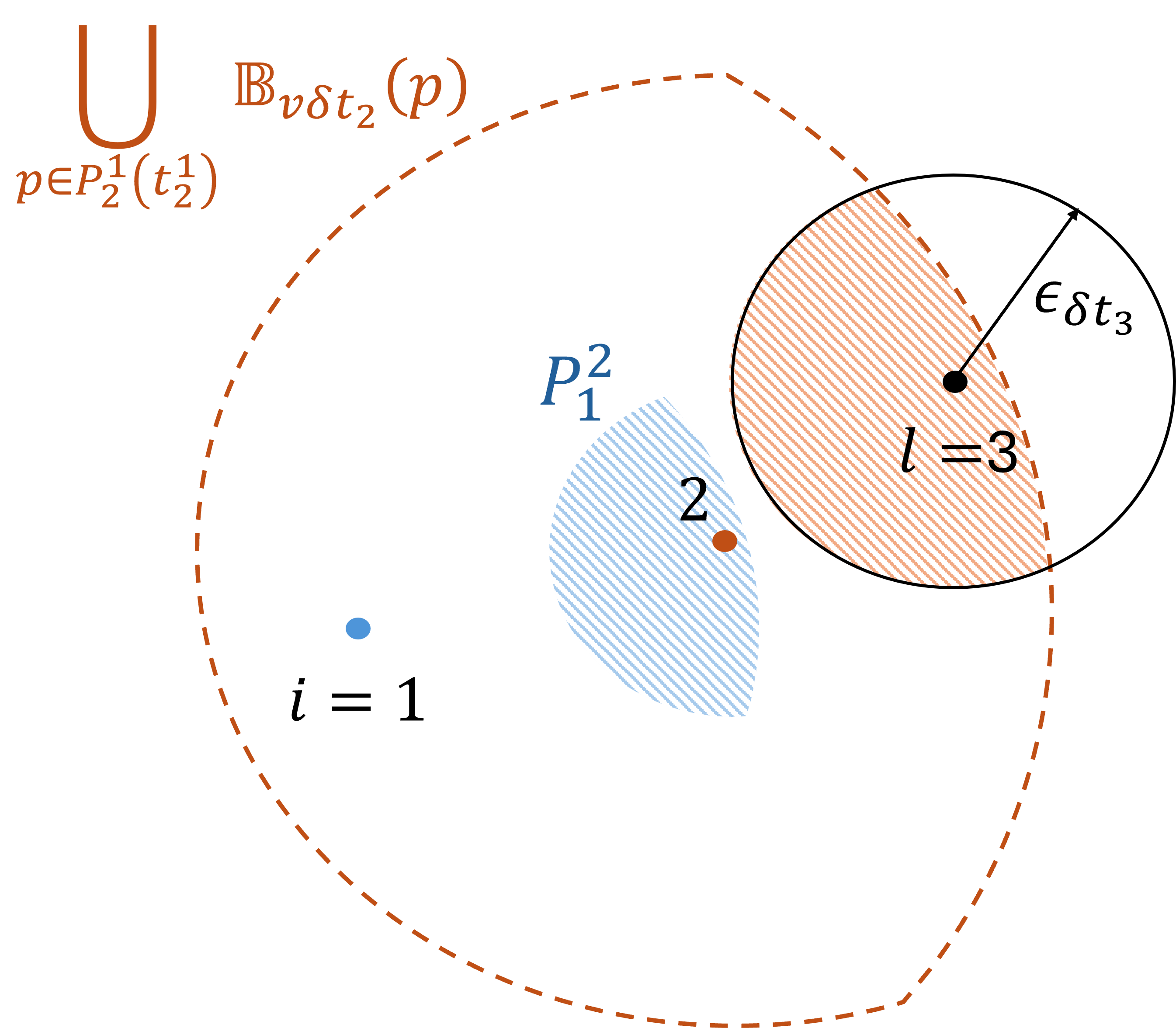}
        \label{fig: pos est error multi}
    \end{subfigure}
    \caption{Position estimate, denoted as the shaded region, for 1-hop and 2-hop neighbors in the presence of time delays.}
    \label{fig: pos est error}
\end{figure}

Recall that all this information is collected at the central node $v_c$ for the sake of making a decision on edge deletion. Hence, we can assume that if $(i, j)\in \mathcal E(0)$ and has not been permitted to be removed by time $t$, it holds that $(i, j)\in \mathcal E(t)$. Hence, denoting $l = v_c$, if $(i, j)\in \mathcal E(t)$, then it holds that 
    $\|p_i^l(t)-p_j^l(t)\|\leq R$.
Hence, such a constraint can be used to shrink the estimated locations of all the nodes in the graph through their connectivity information. 

\subsection{Cost estimation}
As noted in Section \ref{sec: Problem}, the decision on whether an edge should be deleted is based on the cost of maintaining the edge, which in turn is a function of the physical distance of its nodes; see \eqref{eq: cost distance}. Let $l = v_c$ be the central node, as defined in Section \ref{sec: Preliminaries} and $p_i^l(t)$ and $p_j^l(t)$ be the estimated locations of the nodes $i, j$ by $l$ at time $t$, with $(i, j)\in \mathcal E(t)$. Based on the discussion in the previous section, the estimated distance $d_{ij}^l(t)$ between the nodes $i, j$ be the node $l$ is given by:
\begin{align}\label{eq: dist estimate ij}
    d_{ij}^l(t) \in D_{ij}^l(t) \coloneqq & \big\{d\; | \; d= \|p_1-p_2\|, \nonumber \\
    & ~ p_1\in P_i^l(t), p_2\in P_j^l(t)\big\}. 
\end{align}
Based on this, given a threshold $\alpha \in (0, 1)$, we can define the risk-score based on the edge-length between the nodes $(i, j)\in \mathcal E(t)$ at a time $t$ (according to the information available at the central node $l$) is greater than $\alpha R$ as:
\begin{align}\label{eq: prob edge delete}
    P_{ij}^{>}(\alpha) \coloneqq \frac{|\{d\in D_{ij}^l(t)\; |\; R\geq d> \alpha R\}|}{|\{d\in D_{ij}^l(t)\; |\; d\leq R\}|}. 
\end{align}
It is easy to see that $ P_{ij}^{>}$ is a decreasing function of the parameter $\alpha$, i.e., $P_{ij}^{>}(\alpha_1)\leq P_{ij}^{>}(\alpha_2)$ for $\alpha_1\geq \alpha_2$, with $P_{ij}^{>}(0) = 1$ and $P_{ij}^{>}(1) = 0$. Since the cost function in \eqref{eq: cost distance} takes a non-zero value only when $p_{ij}(t)\geq \rho_m R$ where $\rho_m\in (0, 1)$, the estimate of a distance smaller than $\rho_m R$ has a zero estimated cost. As a result, the central node can calculate the estimated upper-bound on the cost for maintaining an edge $(i, j)\in \mathcal E(t)$ at time $t$:
\begin{align}\label{eq: max cost est}
    c_{ij}^l(\rho_m,t) = c_{max} R(1-\rho_m) P_{ij}^{>}(\rho_m),
\end{align}
for a given $\rho_m\in (0, 1)$ and $c_{max}>0$ in the cost-function \eqref{eq: cost distance}.

\subsection{Decision-making}
Based on the estimated cost in \eqref{eq: max cost est}, the central node $l= v_c$ needs to make a decision on the edge $(i, j)\in \mathcal E(t)$ to be deleted as well as a set of new edges $\mathcal E_f$ so as to minimize the total cost of maintenance of the graph $\mathcal G(t)$ at a time $t$ and keeping the diameter of the new graph within the required bounds. The new edges can be proposed for nodes that are \textit{close enough} to each other as estimated by the central node $l$. 
The cost-parameter $\rho_m$ in \eqref{eq: cost distance} dictates whether an edge constitutes a cost, and hence, we define the confidence-score of a new edge formation between nodes $i, j$ such that $(i, j)\notin \mathcal E(t)$ at time $t$:
\begin{align}\label{eq: prob new edge}
   P_{ij}^{<}[\rho_M] \coloneqq \frac{|\{d\in D_{ij}^l(t)\; |\; d< \rho_m R\}|}{| D_{ij}^l(t)|}.
\end{align}

\begin{algorithm}
\caption{Hybrid topology control 
}
\label{algo: main algo}
\KwIn{$\mathcal G$, $\overline{\mathcal G}$, central node $v$, position estimate regions $P_i^v$, thresholds 
$p, \tau_D, \bar c, \delta, C,  \Delta$}
\KwOut{New graph $\mathcal G'$, new central node $v'$}


\BlankLine
\Comment{ Part A: Find high-cost edges to be dropped}
${\mathcal E}_d = \emptyset$\;

\For{$i\in \mathcal V, j\neq i$}{
        Compute cost estimate $c_{ij}^v$\;
        \Comment{Check cost and connectivity constraints}
        \If{$c_{ij}^v\geq \bar c$ \textnormal{and} $\lambda_2(L(\mathcal E\setminus (i, j)))>0$}{
            $e_d = (i, j)$\;
            $\hat {\mathcal E}' = (\mathcal E\setminus ({\mathcal E}_d\cup e_d))$\;
    $\hat{\mathcal G} = (\mathcal V, \hat{\mathcal E}')$\;
    \Comment{Check total cost of new graph}
    $c_{new} = \sum\limits_{(i, j)\in \hat{\mathcal E}'(t)}c_{max} R(1-\rho_m) P_{ij}^{>}(\rho_m)$\;
    \Comment{Check diameter of new graph}
    $\hat D = D(\hat{\mathcal G})$\;
    \If{$c_{new}\leq C-\delta$ \textnormal{and} $\hat D\leq \tau_D$}{
    $\mathcal E_d = \mathcal E_d\cup e_d$\;
        }
    }
}

$\mathcal E' = \mathcal E\setminus\mathcal E_d$\;

\BlankLine
\Comment{ Part B: Find useful edges to be added }
${\mathcal E}_f = \emptyset$\;
\For{$(i, j)\notin \mathcal E$}{
\Comment{Check closeness of the nodes}
\If{$P_{ij}^{<}[\rho_M] \geq 1-p$}{


${\mathcal E}_f = {\mathcal E}_f\cup (i, j)$
}
}

\Comment{Central node waits for confirmation about new edge formation}

$\hat{\mathcal E}\subset \mathcal E_f$: Edges that sent confirmation on addition\;

$\mathcal E' = (\mathcal E\setminus\mathcal E_d)\cup \hat{\mathcal E}_f$\;
$\mathcal G' = (\mathcal V, \mathcal E')$\;

\BlankLine
\Comment{ Part C: Re-elect central node }

$v' = \argmin\limits_{v\in \mathcal V}\max\limits_{j\in \mathcal V}d_{\mathcal G'}(v, j)$

\BlankLine

Return: $\mathcal G', v'$

\end{algorithm}

\subsection{Decision implementation}
The decision on edge deletion, as well as new edges to be added, is broadcast by the central node $l$. A drop or add order sent to nodes $i$, $j$ reaches the respective nodes after
$\delta t_i$ and $\delta t_j$ seconds. The nodes involved in the edge deletion can simply disconnect. However, the nodes involved in new edge formation over the base graph must send a confirmation to the central node, which takes an additional communication delay. As a result, the central node updates $\mathcal G\to \mathcal G'$ only after receiving the confirmation of new edge formation, communication delay for which is dictated by $2~\text{rad}(\mathcal G)$ in the worst case. As a result, the central node cannot be sure whether the proposed new edges will definitely form when making a decision on edge deletion.
Hence, a simultaneous edge deletion and addition cannot be performed to create the new base graph. As a workaround, we propose a sequential mechanism where the central node makes the decision on edge deletion based on the risk-score \eqref{eq: prob edge delete} and diameter constraint,
followed by a decision on new edge formation based on confidence-score \eqref{eq: prob new edge},
sequentially. The intermediate graph before the execution of the central node's decision is denoted as $\mathcal G_r = (V, \mathcal E_r)$. Once the decision reaches the concerned nodes and is executed successfully, the graph known to the central node $\mathcal G'$ and the actual graph $\mathcal G_r$ become the same (see Section \ref{sec: sim} for more details). The overall algorithm is described in Algorithm \ref{algo: main algo}.

\subsection{New central node selection}
Upon deletion of the edges $\mathcal E_d$ and addition of edges $\mathcal E_f$, it is possible that another node $v_c'\neq v_c$ becomes the central node for the new graph $\mathcal G'(t)$. Based on the adjacency information $\mathcal E'(t)$ of the new graph $\mathcal G'(t)$, the \textit{current} central node $v_c$ computes the new central node through $v_c' = \argmin_{v\in \mathcal V}e_{\mathcal G'(t)}(v)$ and broadcasts this information along with its decision $(\mathcal E_d, \mathcal E_f)$. Computation of the new central node requires computing shortest paths between each pair of nodes in the new graph $\mathcal G'$. A fundamental property of shortest paths under edge deletion is that distances
can only increase, and only for vertex pairs for which every shortest path in $\mathcal G$ used the deleted edge. This observation is standard in the literature on dynamic shortest-path algorithms
\cite{frigioni2000dynamic,demetrescu2006dynamic}. Consequently, for any vertex pair $(u,v)$ such that there exists at least one
shortest $u$--$v$ path in $\mathcal G$ avoiding $e$, we have $d_{G'}(u,v) = d_G(u,v)$. Using precomputed shortest-path counts $\sigma_{uv}$, one can test in constant time whether all shortest paths between $u$ and $v$ traverse a given edge $e$ via standard path-decomposition arguments \cite{brandes2001betweenness}. Thus, only vertices $u$ for which at least one farthest vertex is affected by the removal of $e$ require further processing. For each such vertex $u$, the updated eccentricity $e_{\mathcal G'}(u)$ can be computed by recomputing single-source shortest paths from $u$ in $\mathcal G'$, while all other eccentricities remain unchanged. As a result, the new ordered set $\mathcal V_e(G')$ follows directly from the updated eccentricities. This approach is an exact specialization of known decremental all-pairs shortest path (APSP) and eccentricity-maintenance techniques
\cite{rodittzwick2012dynamic,italiano2012dynamic}. The worst-case running time for such updates remains quadratic, which is unavoidable due to known lower bounds for the dynamic diameter and eccentricity problems \cite{roditywilliams2013hardness}. However, in practice, the number of affected vertices is typically small, making incremental recomputation substantially faster than recomputing all-pairs shortest paths from scratch.



\section{Simulation results}\label{sec: sim}
We validate the performance of the proposed hybrid topology control framework (termed Method A in the results) through a series of numerical simulations. We consider graphs with 20 nodes, a communication delay of $0.5$s for one-hop communication, and a diameter bound of $\tau_D = 8$. The central node makes decisions every $40$ steps. We consider the following methods as baselines for comparisons:
\begin{itemize}
    \item \textbf{Method B}: An MST-based centralized framework \cite{kershenbaum1972computing} without any communication delays or diameter bound. This is an idealized baseline to compute the minimum possible accumulated cost for edge maintenance.
    \item \textbf{Method C}: An MST-based centralized framework \cite{kershenbaum1972computing} with no time delays but with the same diameter bound of $\tau_D = 8$. 
    \item \textbf{Method D}: A distributed approach where we use the proposed decision-making mechanism but with a fixed central code. 
\end{itemize}
Here, we use $\mathcal E_s(t) = \{(i, j)\in \mathcal E(t)\; |\;  c_{ij}>0\}$ as the set of ``stressed" edges whose cost is non-zero. The cumulative cost is defined as $C = \sum_t\sum_{(i, j)\in \mathcal E(t))} c_{ij}(t)$. Figure \ref{fig: comparison all combined} plots the performance of all the methods for a randomly chosen initial configuration. The left figures in each row plot the graph at the final time instant, highlighting the central node $l$ with a star, the positional uncertainty regions $P_i^l$ for each node $i\neq l$ with dotted circles, and the graph edges. The right figure plots the number of edges at each time instant in $E$ as well as the number of stressed edges in $E_s$, with a red-dotted line denoting the desired upper-bound on the number of stressed edges.  Noting that the \textit{real} graph, with edges denoted as $E_r$, after a decision has been made by the central node, might be different from the graph known to the central node, we plot the number of edges in $E_r$ also in Figure \ref{fig: comparison all combined}. The two centralized approaches lead the graph to form a linear minimum spanning tree (MST) in the first decision, primarily because there are no diameter constraints. However, the linear graph achieved by Method B might not be desirable in many applications. The centralized approach with time delay (i.e., Method C) also leads to an MST, but because of the enforced diameter constraint, the resulting MST is not linear and satisfies the maximum diameter requirement. Finally, the distributed method with a fixed central node (Method D), which, in some sense, is a pseudo-distributed method since it still has a single point of failure, exhibits behavior similar to that of the proposed approach.

\begin{figure}[h]
    \centering
    \includegraphics[width=0.95\linewidth]{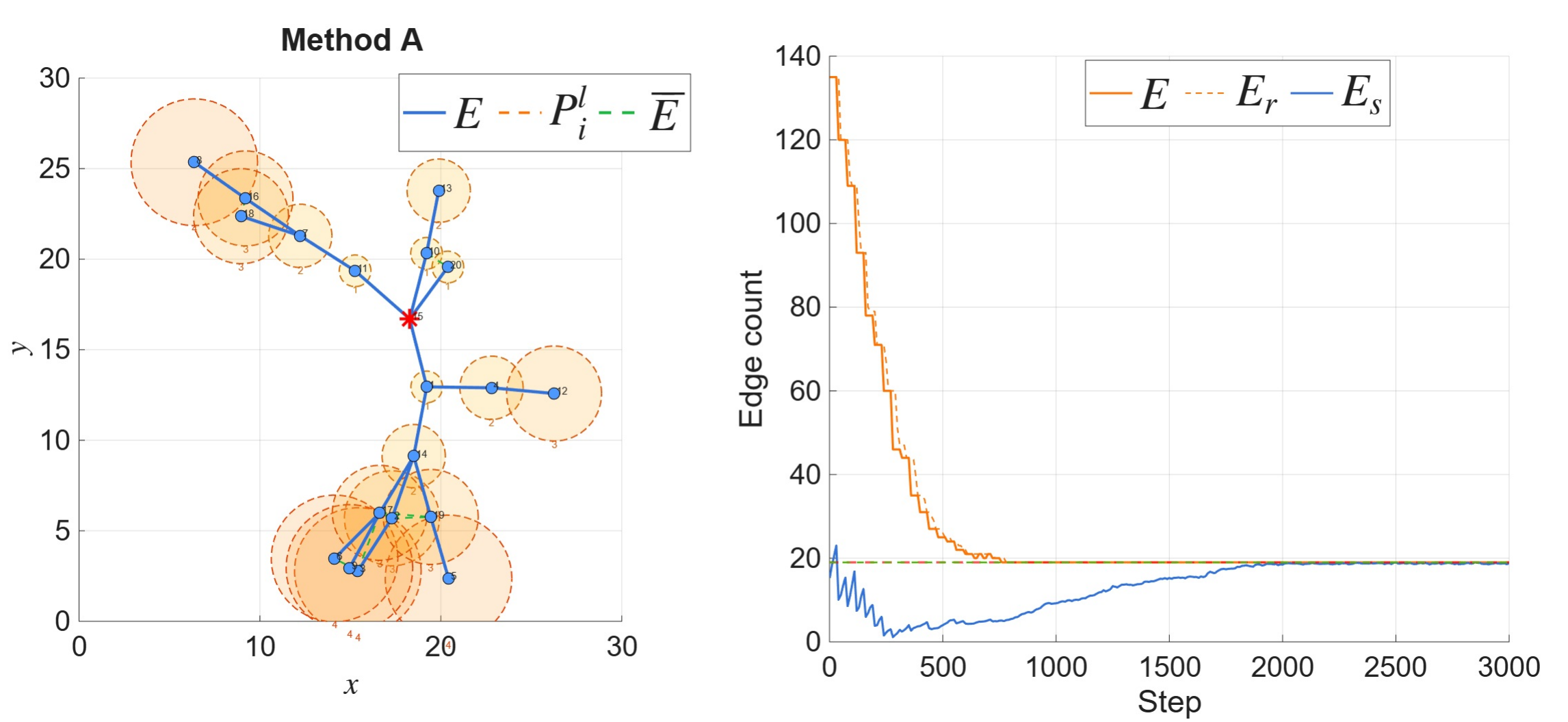}
    \includegraphics[width=0.95\linewidth]{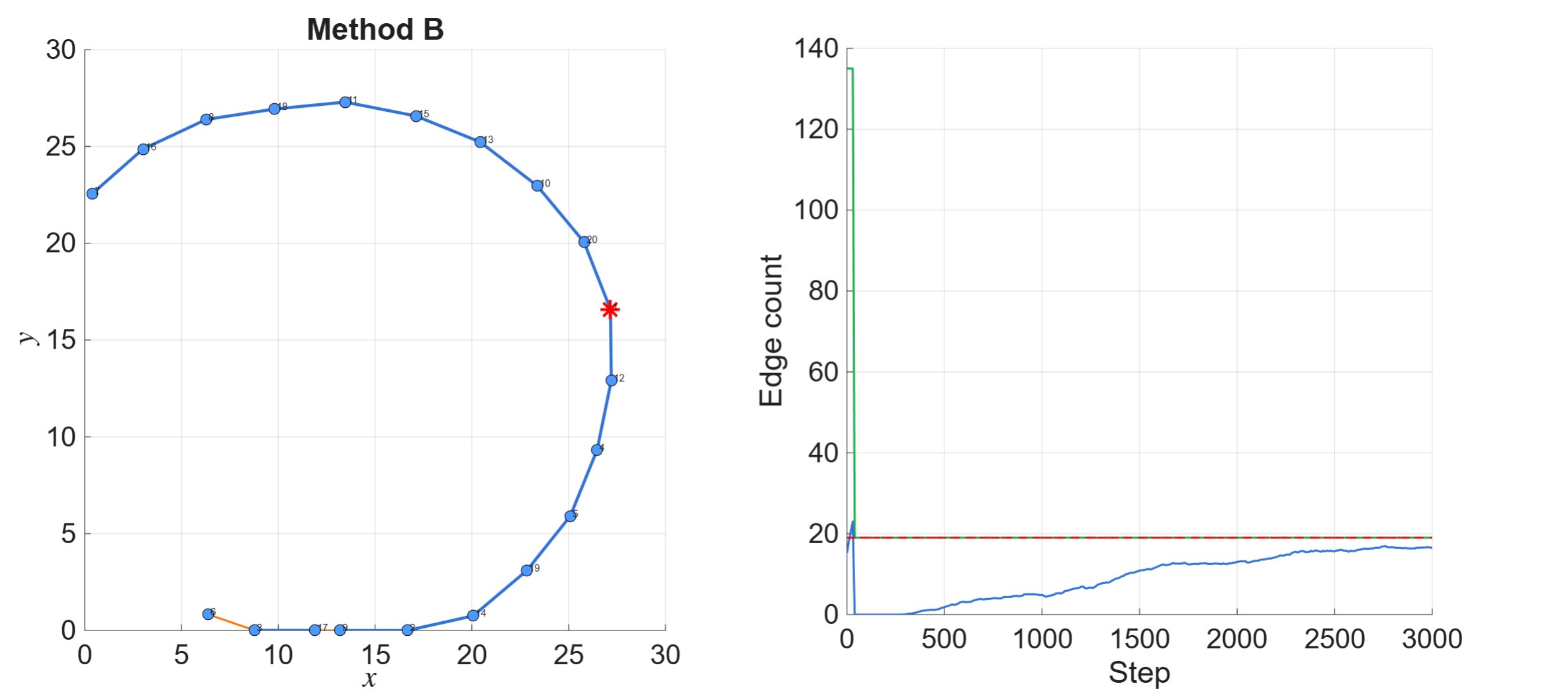}
    \includegraphics[width=0.95\linewidth]{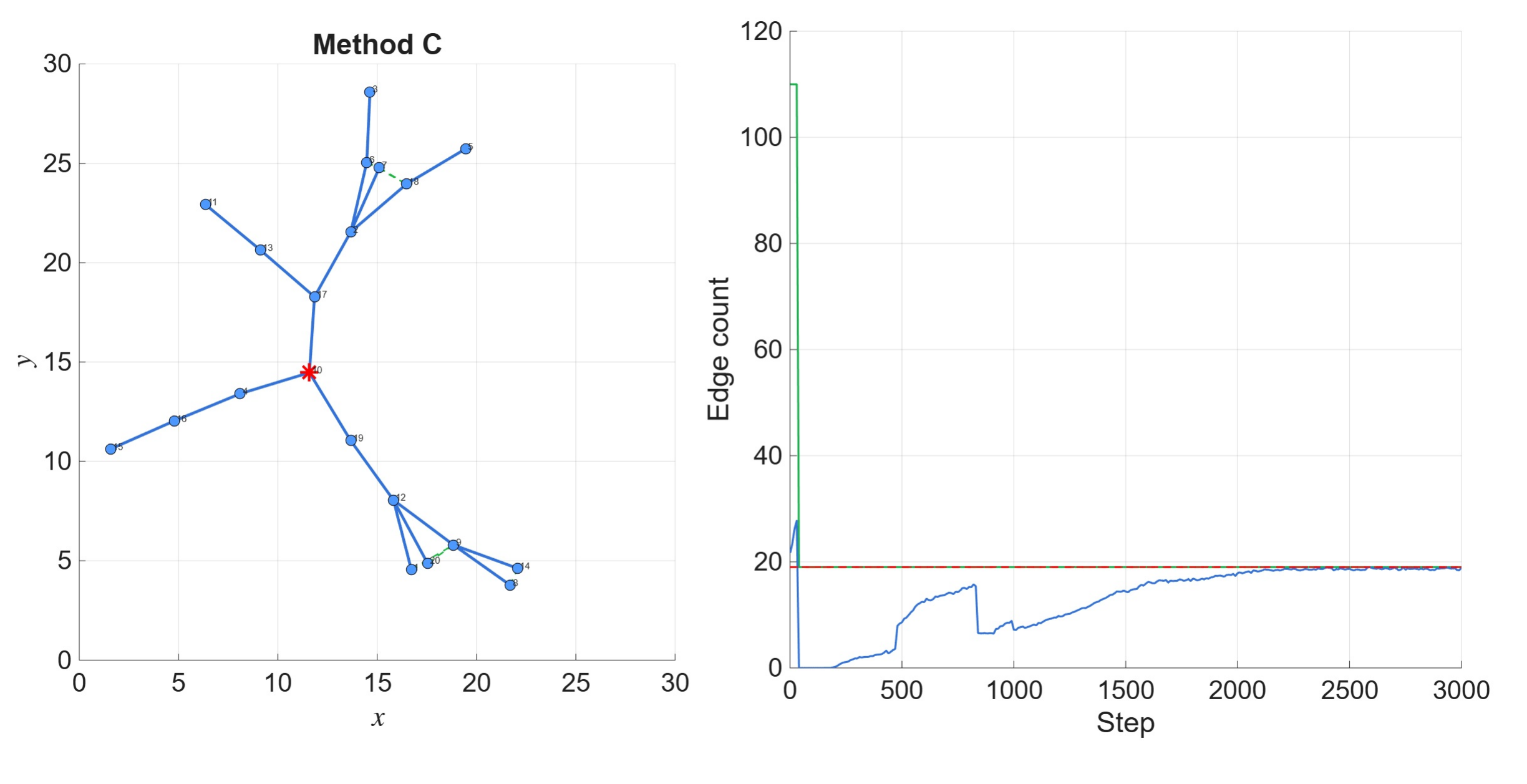}
    \includegraphics[width=0.95\linewidth]{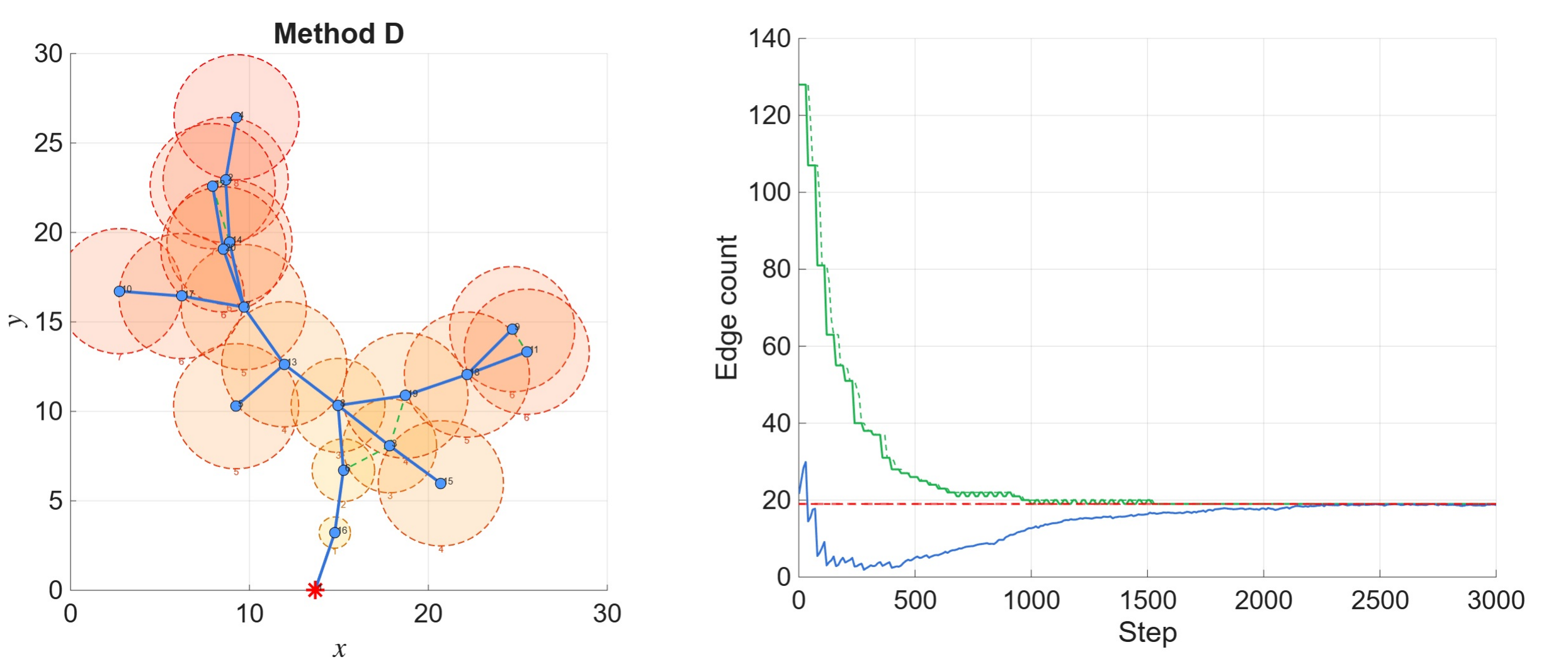}
    \caption{Simulation results for \textbf{A: proposed method}, \textbf{B: centralized method with no delays and no diameter limit}, \textbf{C: centralized method with no diameter limits}, and \textbf{D: distributed method with fixed central node} with 20 nodes. In each row, the left figure illustrates the final graph after 3000 iterations, and the right figure plots the number of edges in the graph and the number of stressed edges at each time step.} 
    \label{fig: comparison all combined}
\end{figure}

Table \ref{tab: cost comparison} lists the cumulative costs and the wall-clock runtime averaged over 1000 runs of the considered methods. As expected, Method A provides a global minimum for the total cost and also has the minimum runtime. However, the proposed method outperforms the other approaches, demonstrating that it accounts for realistic communication delays and positional uncertainties while remaining on par with idealistic methods. Furthermore, when the MST-based method is subject to the diameter constraint, its runtime increases significantly, highlighting the efficiency of the proposed method.




\begin{table}[h]
\setlength{\tabcolsep}{10pt} 
\renewcommand{\arraystretch}{1.25} 
    \begin{center}
        
    \begin{tabular}{|c|c|c|c|}
        \hline
        \textbf{Method} & \textbf{Cost} & \textbf{Time} \\
        \hline \hline
        Method A & 2069.3 & 0.4s  \\
        \hline
        Method B & 1572.9 & 0.1s \\
        \hline
        Method C & 2105.5 &  1.4s\\
        \hline
        Method D & 2082.9 & 0.4s \\
        \hline
    \end{tabular}
    \caption{Cumulative costs and wall-clock time for the proposed and baseline methods averaged over 1000 runs for graphs with $|\mathcal V| = 20$ and $\tau_D = 8$. }
    \label{tab: cost comparison}
    \end{center}
\end{table}

Noting that Method B is computationally very efficient for the case of $|\mathcal V|=20$, we benchmark the proposed method against it as we vary the number of nodes for a fixed diameter constraint $\tau_D = 8$, see Table \ref{tab:runtime for different N}. We observe that the computation time increases almost linearly for the efficient MST-based algorithm, whereas that of the proposed method grows superlinearly. We do note that the proposed \textit{greedy} algorithm has not been optimized for performance, and for up to 50 nodes, the per-decision time of $\sim 0.04$s is acceptable. One of the reasons for the long computation time is that we use the Monte Carlo method to evaluate the risk-score \eqref{eq: prob edge delete} and the confidence-score \eqref{eq: prob new edge}. Future work involves improving the algorithm's computational performance and optimizing its implementation to minimize computational overhead.


\begin{table}[h]
\renewcommand{\arraystretch}{1.25}
    \centering
    \begin{tabular}{|c|c|c|c|c|c|}
        \hline
        \backslashbox{\text{Method}}{$|\mathcal V|$} & $20$ & $30$ & $40$ & $50$& $60$ \\
         \hline
         \hline 
         Method A & 0.36s & 0.78s &  1.81s & 4.02s & 8.21s \\
         \hline
         Method B & 0.06s & 0.09s &  0.14s & 0.21s &  0.32s \\
         \hline
    \end{tabular}
    \caption{Wall-clock time comparison for the proposed Method A and the baseline Method B over 100 runs with 100 decisions in each run for varying $|\mathcal V|$ and fixed diameter bound of $\tau_D = 8$.}
    \label{tab:runtime for different N}
\end{table}

Finally, we perform simulation experiments on a graph with a fixed size, i.e., $|\mathcal V| = 50$, and vary the required diameter bound $\tau_D$ to compare the performance, in terms of the cumulative cost and the runtime, of the proposed method and the centralized baseline Method C that accounts for the diameter bound, see Table \ref{tab:runtime for different D}. The cumulative cost decreases with increasing $\tau_D$, consistent with the intuition that a larger diameter allowance allows the central node to make better decisions, reducing the overall cost of edge maintenance. As the ratio $\frac{\tau_D}{|\mathcal V|}$ decreases, the cumulative cost increases for both the proposed method and the baseline Method C. However, the statistics show that the proposed method outperforms the baseline in each case.

\begin{table}[h]
\renewcommand{\arraystretch}{1.25}
\centering
\begin{tabular}{|l|cc|cc|cc|}
\hline
  & \multicolumn{2}{c|}{$\tau_D = 5$}                     & \multicolumn{2}{c|}{$\tau_D = 10$}                    & \multicolumn{2}{c|}{$\tau_D = 15$}                    \\ \hline
                         & \multicolumn{1}{c|}{Time} & Cost & \multicolumn{1}{c|}{Time} & Cost & \multicolumn{1}{c|}{Time} & Cost \\ \hline
{Method A} & \multicolumn{1}{c|}{  \textbf{5.2s } }          &       \textbf{5922.9}        & \multicolumn{1}{c|}{\textbf{4.2s}}              &          \textbf{4107.7}     & \multicolumn{1}{c|}{\textbf{3.9s}}              &       \textbf{3410.9}        \\ \hline
{Method C}        & \multicolumn{1}{c|}{32.8s}              &     6991.4       & \multicolumn{1}{c|}{57.6s}              &      4814.2        & \multicolumn{1}{c|}{33.9s}              &       {4314.1}       \\ \hline
\end{tabular}
\caption{Wall-clock time and cumulative cost comparison for the proposed Method A and the baseline Method C over 10 runs with 100 decisions in each run for varying diameter bound $\tau_D$ and fixed number for nodes $|\mathcal V| = 50$.}
    \label{tab:runtime for different D}
\end{table}

The runtime of the proposed method also decreases as the diameter bound $\tau_D$ increases, which makes sense as the number of possibilities of edge deletion increases with an increase in $\tau_D$, allowing the central node to make a decision faster. For the baseline Method C, the runtime does not follow a monotonic pattern as this method consists to two components: an MST algorithm and a diameter-constraining step that utilizes APSP. Its computational cost is dominated by the second component, which evaluates candidate edges on the current diameter path; each evaluation recomputes APSP distances. When the diameter bound is loose (e.g., $\tau_D = 15$), the initial MST typically satisfies the constraint, and the diameter-check step terminates quickly. When the bound is very strict (e.g., $\tau_D = 5$), the constraint is often infeasible with the available communication edges. As a result, the diameter-enforcing search stalls early, and the method falls back to a breadth-first search (BFS) tree construction. The runtime peak occurs at intermediate bounds (e.g., $\tau_D = 10$), where the constraint is tight but feasible, causing many candidate evaluations.

\section{Conclusions}\label{sec: conclusions}
We proposed a novel hybrid topology control framework for a group of mobile robots under the presence of unknown uncertainties and communication delays. The proposed approach is a hybrid of centralized and distributed frameworks, where the central node decides which edges to delete to reduce the cost of maintaining the communication network, and which new edges to add to keep the graph diameter within given limits, and the central node can change based on the graph topology. We presented an efficient mechanism for estimating the positions of nodes in a graph in the presence of bounded uncertainties and communication delays, using information collected from all nodes, and used it to estimate the cost of edge maintenance. Finally, we illustrated, through numerical case studies, that the proposed approach performs on par with idealized methods without making such strong, unrealistic assumptions. 

As stated in the previous section, future work will focus on optimizing the decision-making framework for computational efficiency. Future work also involves extending the framework to a truly distributed one, where each node can safely decide whether to drop or add edges based on time-delayed information. 

\bibliographystyle{IEEEtran}
\bibliography{myreferences}

\end{document}